\documentclass[11pt, reqno]{amsart}
\usepackage{amsthm,amsmath,amsfonts,amssymb,color}
\usepackage[bookmarks]{hyperref}
\reversemarginpar

\newtheorem{thm}{Theorem}[section]
\newtheorem{prop}[thm]{Proposition}
\newtheorem{cor}[thm]{Corollary}
\newtheorem{lemma}[thm]{Lemma}
\newtheorem{defn}[thm]{Definition}

\newtheorem{preremark}[thm]{Remark}
\newenvironment{remark}{\begin{preremark}\rm}{\medskip \end{preremark}}

\numberwithin{equation}{section}

\newcommand{\norm}[1]{\left\Vert#1\right\Vert}
\newcommand{\abs}[1]{\left\vert#1\right\vert}

\newcommand{\R}{\mathbb R}

\DeclareMathOperator{\Vol}{Vol}

\newcommand{\grad} {\nabla}

\newcommand{\dd} {\mathrm{d}}

\DeclareMathOperator{\supp}{supp}
\DeclareMathOperator{\dv}{div}
\DeclareMathOperator{\curl}{curl}

\DeclareMathOperator{\Ric}{Ric}

\def\H{\mathbb H^{N}(-a^{2})}
\def\be{\begin{equation}}
\def\ee{\end{equation}}
\def\qand{\quad\mbox{and}\quad}
 
\newcommand{\tV}{\textbf V}

\begin{document}
\title[Hodge decomposition]{Hodge decomposition of the Sobolev space $H^1$ on a space form of nonpositive curvature}

\author[Chan]{Chi Hin Chan}
\address{Department of Applied Mathematics, National Chiao Tung University,1001 Ta Hsueh Road, Hsinchu, Taiwan 30010, ROC}
\email{cchan@math.nctu.edu.tw}

\author[Czubak]{Magdalena Czubak}
\address{Department of Mathematics\\
University of Colorado Boulder\\ Campus Box 395, Boulder, CO, 80309, USA}
\email{czubak@math.colorado.edu}

\author[Pinilla Suarez]{Carlos Pinilla Suarez}
\address{Department of Mathematics\\
University of Colorado Boulder\\ Campus Box 395, Boulder, CO, 80309, USA}
\email{Carlos.PinillaSuarez@colorado.EDU}

\begin{abstract}
The Hodge decomposition is well-known for compact manifolds.  The result has been extended by Kodaira to include non-compact manifolds and $L^2$ forms.  We further extend the Hodge decomposition to the Sobolev space $H^1$ for general $k$-forms on non-compact manifolds of nonpositive constant sectional curvature. As a result, we also obtain a decomposition on $\R^N$.
\end{abstract}
%\date{\today}
\maketitle

\tableofcontents

 \section{Introduction}
Hodge decompositions are widely studied and have many applications.  The main idea is to take an object, say a tensor, and decompose it into a sum of what can be viewed as canonical pieces.   The Hodge decomposition is well-known for compact manifolds.  If we let $\Lambda^k$ denote the set of smooth differential forms on a Riemannian manifold $M$,  and $\mathcal H^k$ denote the harmonic forms on $M$, then we have

\begin{thm}[Hodge decomposition, compact manifolds]  Let $(M,g)$ be a compact Riemannian manifold without boundary.  With respect to the metric $g$, we have 
\[
\Lambda^k=\dd\Lambda^{k-1} \oplus  \dd^\ast \Lambda^{k+1}\oplus \mathcal H^k.
\]
\end{thm}
Here, $\dd$ denotes the exterior derivative, and $\dd^\ast$ its adjoint (these are reviewed in Section \ref{Prelim}).  This means that for example, when $k=1$, any smooth $1$-form $\alpha$ on $M$ can be uniquely decomposed as
\[
\alpha=\dd f+\dd^\ast \omega + h,
\]
where $f$ is a function, $\omega$ is a $2$-form and $h$ is a harmonic $1$-form.

This result has been extended by Kodaira to include $L^2$ $k$-forms on non-compact manifolds \cite{Kodaira}.  
 Let $\Lambda_c^k(M)$ denote the space of all smooth $k$-forms with compact support on $M$, then the Hodge-Kodaira decomposition is 
 \begin{thm}[Hodge-Kodaira decomposition for non-compact manifolds]\cite{Kodaira}\label{Kodaira}  Let $(M,g)$ be a complete Riemannian manifold without boundary.  Then 
\[
L^2(\Lambda^k)=\overline{\dd\Lambda^{k-1}_c}^{L^2}\oplus \overline{\dd^\ast \Lambda^{k+1}_c}^{L^2}\oplus \mathcal H^k,
\]
where $\mathcal H^k$ denotes the harmonic $L^2$ $k$-forms on $M$.
\end{thm}
Note, for example, $\overline{\dd\Lambda^k_c}^{L^2}$ means the closure in the $L^2$ norm of the image of the operator $\dd$ acting on the smooth $k$-forms, with compact support on $M$.

 Kodaira used the functional analysis approach following Weyl \cite{Weyl}. Due to the important contributions of de Rham \cite{deRham_thesis, DeRhamEng}, Hodge  \cite{Hodge, HodgeBook}, Weyl \cite{Weyl, Weyl_Hodge} and Kodaira 
 \cite{Kodaira}  such decompositions, when referred to, can be seen to include, besides Hodge, the names of any of these mathematicians.
 
 We also mention the result of Gromov \cite{Gromov91} of what is called the strong $L^2$ decomposition under the spectral gap assumptions. 
%  Besides studying the decomposition of the $L^2$ space, 
  In addition, one can consider the decomposition of $L^p$ spaces for $p\neq 2$.  See for example \cite{Scott, XDLi09, Amar17}. Besides the perspective of the study being taken to be either $L^2$ or general $L^p$, compact or non-compact manifolds, one can investigate manifolds and domains with or without boundaries, general Sobolev spaces, weighted and unweighted; and further take the decompositions regarding other elliptic operators \cite{BergerEbin, Friedrichs, Morrey56, MorreyEells, McOwen79, Cantor81, Schwarz95, MitreaTaylor_book}.  
 
 In spite of these vast developments, it is to our surprise that we have not found anywhere the decomposition written for the Sobolev space $H^1$ (un-weighted) on a non-compact manifold without boundary. Hence the goal of this article is to provide a relatively simple proof, and in the process, to give an expository review of the proof of the Hodge-Kodaira decomposition in $L^2$.

So now, if we would like to do the Hodge decomposition of the Sobolev space $H^1$, then comparing to Theorem \ref{Kodaira}, it is natural to expect to obtain the following decomposition
\be\label{hodgep}
H^1(\Lambda^k)=\overline{\dd\Lambda^{k-1}_c}^{H^1}\oplus \overline{\dd^\ast \Lambda^{k+1}_c}^{H^1}\oplus \mathcal H^k,
\ee
where now we take the closure in the Sobolev space $H^1$, and $\mathcal H^k$ are harmonic $k$-forms in $H^1$.  However, since one can show the harmonic $k$-forms in $L^2$ are actually in $H^1$ (see Section \ref{harmon}), so in \eqref{hodgep}, $\mathcal H^k$ can denote the harmonic $k$-forms in $L^2$ as before.
 
We prove the following theorem.
\begin{thm}[Hodge Decomposition in $H^1$ for $k$-forms]\label{Thm1}
Let $a\geq 0$, and let $\Lambda^k$ be the space of differential $k$-forms over $\H$, then \eqref{hodgep} holds.  Moreover, if $\alpha \in H^1$, then we have
\be\label{hodgepal}
\alpha=\dd \beta+\dd^\ast \omega+\gamma,
\ee
where $\dd \beta$ is in the $H^1$ closure of $\dd\Lambda^{k-1}_c$, analogously $\dd^\ast \omega \in \overline{\dd^\ast \Lambda^{k+1}_c}^{H^1}$, and $\gamma$ is a harmonic $L^2$ $k$-form.
\end{thm}

 Of special interest in PDE theory is the case of $1$-forms.  The reason for this is that on a Riemannian manifold, $1$-forms are naturally identified with vector fields, which in turn, relate to the solutions of systems of PDE.   The identification between  $1$-forms and vector fields  is accomplished using the Riemannian metric (see below Section \ref{Mgdef}).

The case of the Hodge decompositions for 1-forms, or equivalently that of the vector fields,  is often called Helmholtz decomposition or Helmholtz-Weyl decomposition, and has applications to fluid mechanics, electromagnetism and the study of boundary value problems. This goes back to the aforementioned work of Weyl \cite{Weyl}, and even further back, to the work of Helmholtz in 1858 \cite{Helmholtz1858}.  Classically, it means writing something as divergence free plus a gradient. For relevant works we refer, for example, to \cite{BS1960, SimaderSohr, Galdi, Schwarz95, Mitrea2D, Mitrea3D, MitreaJFA, MitreaTaylor_book}.
 
 We allow $a=0$ in Theorem \ref{Thm1} as then we can recover the Euclidean case, for which of course, there are no nontrivial harmonic $k$-forms in $H^1$, and in the case of $k=1$, the decomposition reduces to the case of Helmholtz-Weyl decomposition.
 
  It is interesting to consider the case, when there are nontrivial harmonic forms present.  If $M=\H$ has a constant negative sectional curvature, by work of Dodziuk \cite{Dodziuk} we know there exist nontrivial $L^2$ harmonic forms of degree $k=\frac N2$, where $\dim M=N$.  In 2D this corresponds to nontrivial harmonic $1-$forms, in $4D$ to nontrivial harmonic $2-$forms, and so on.  This is a reason we consider the negative curvature case as we know there are nontrivial harmonic forms present.  In addition, this is a natural follow-up to the previous work of the first two authors.  
 
 In particular, \cite{CC13, CC15} studies $1$-forms that are divergence free in $H^1$ and shows they can be decomposed as harmonic $L^2$ forms and limits in $H^1$ of divergence free compactly supported $1$-forms.  More precisely, consider
\[
\widetilde {\textbf{V}}=\{u\in H^1(\Omega): \dd^\ast u=0\},
\]
and
\be\label{tV}
 \textbf{V} = \overline{\Lambda_{c,\sigma}^1(\Omega)}^{H^1},
\ee
where $\Lambda^1_{c,\sigma}$ denotes, smooth compactly supported and divergence free forms on a domain $\Omega$.
 %these spaces are not the same unless $N=2$. 
  It was observed by Heywood \cite{Heywood} that whether or not these spaces coincide is related to having nonunique solutions to the stationary Stokes and Navier-Stokes equations.  These spaces are for example the same for $\Omega=\R^n$, but in \cite{CC13}, the first two authors showed that these spaces are not the same on a hyperbolic space when $N=2$, and in fact

\be\label{hodge1}
\widetilde {\textbf{V}}=\textbf V \oplus \mathcal H,
\ee
which could explain the non-uniqueness phenomenon presented in \cite{CC10} (see also \cite{KhesinMisiolek, Lichtenfelz}) (when $N\geq 3$, $\widetilde {\textbf{V}}=\tV$ \cite{CC15}).  From the point of view of PDE, the definition of the space $\tV$ as given by \eqref{tV} is convenient to work with, but it can be shown as a corollary to Theorem \ref{Thm1} that $\tV= \overline{\dd^\ast \Lambda^{k+1}_c}^{H^1}$.  We give a \emph{constructive} proof of that fact for $N=2$. 

\begin{thm}\label{Thm2}
Consider $H^1(\mathbb H^2(-a^2))$, then 
\[
 \widetilde {\textbf{V}}= \overline{\Lambda_{c,\sigma}^1}^{H^1}\oplus \mathcal H^1=\overline{\dd^\ast \Lambda^2_c}^{H^1}\oplus \mathcal H^1.
\]
\end{thm}
It follows that the statement of the equation \eqref{hodge1} is a subset of the Hodge decomposition of the space $H^1$ for $1$-forms that are divergence free.  To obtain the full Hodge decomposition for $1$-forms it remains to include the limits of the differentials in the $H^1$ norm.  Hence this article can be viewed as completing this task and moreover extending the Hodge Decomposition to any $k$-form in $H^1$.

Weyl's proof in \cite{Weyl} was for the Helmholtz decomposition of the vector fields  in $L^2(\R^3)$, and relied on the Hilbert space structure of $L^2$.  The application was the study of boundary value problems in potential theory.   We follow the method of Weyl, the method of orthogonal projections, in this article.  We review the proof of the $L^2$ decomposition to motivate what is needed in the $H^1$ case.  In particular, the proof in $H^1$ does not directly follow from the statement of the $L^2$ decomposition even though $H^1$ is a subspace of $L^2$.  This is due to $H^1$ having its own inner product, and not just the $L^2$ inner product.  This is explained more in Section \ref{mainidea}.  The main tool in the proof is the Bochner-Weitzenb\"ock formula for $k$-forms, which is more complicated for $k\neq 1$. However, if we assume constant sectional curvature, then the formula simplifies considerably (See Section \ref{BWPsection}). In addition, we can obtain an explicit estimate of an $\dot H^1$ norm of a harmonic $L^2$ form.

The article is written in an expository manner as the hope  is that it can be readable both to the geometers and PDE theorists.

\subsection{Organization of the paper}
 
In Section \ref{Prelim}, we introduce tools to be used throughout this work. More specifically, we give some definitions from Riemannian geometry, define Sobolev spaces on Riemannian manifolds along with the definitions of weak derivatives.  We also review the Hodge $\star$ operator, $\dd^*$, and the notion of currents.

We give a careful discussion of the  Bochner-Weitzenb\"ock formula in Section \ref{BWPsection}, and in Section \ref{harmon} we show $L^2$ $k$-forms belong to $H^1$.

Section \ref{Hdecomp} is dedidcated to the proof of Theorem \ref{Thm1}.  We begin with the review of the $L^2$ Hodge decomposition.  Finally, Section \ref{ProofThm2} proves Theorem \ref{Thm2}.

 \subsection{Acknowledgements}
We would like to thank Michael Struwe for his observation that significantly simplifies the proof of showing that $L^2$ implies $H^1$ in the case of harmonic $1$-forms that the first two authors had in \cite{CC10}.  One can just integrate by parts instead.  This simple yet insightful observation allows us to give an elegant proof for higher order degree forms as presented in Section \ref{harmon}.
 
 C. H. Chan is partially supported by a grant from the  Ministry of Science and Technology of Taiwan (107-2115-M-009 -013 -MY2). M. Czubak is partially supported by a grant
from the Simons Foundation \# 585745.

\section{Preliminaries}\label{Prelim}
\subsection{Definitions from Riemannian geometry}\label{Mgdef}
Here we establish notation and recall some basic notions from Riemannian geometry.  In the rest of this paper, unless said otherwise, $M$ is used to denote an $N$-dimensional, complete, simply connected Riemannian manifold of constant sectional curvature $-a^2$, without boundary.  By the Cartan-Hadammard theorem, $M$ is non-compact. We let $a\geq 0$, so $M$ could be $\R^N$. 

 Let $g$ be the Riemannian metric on $M$.   The identification of vector fields and $1$-forms is done using the metric, and the so-called musical isomorphisms (lowering/raising indices).  Indeed, if $u$ is a vector field, then we can define a $1$-form $u^\flat$, by 
 \[
 u^\flat(\cdot) = g(u, \cdot).
 \]
If we write $u$ in local coordinates, as $u=u^i \partial_{x^i}$, then $u^\flat=g_{ij}u^j \dd x^i$, where $g_{ij}$ is the $i, j$ entry of the metric $g$ in coordinates, and we sum over repeated indices.  Similarly, if $\omega$ is a $1$-form, then a corresponding vector field is given by $\omega^\sharp$ and defined (implicitly) by
\[
\omega(\cdot)=g(\omega^\sharp, \cdot),
\]
or in coodinates
\[
(\omega^\sharp)^i=g^{ij}\omega_j,
\]
with $(g^{ij})$ being now the inverse of $g$.
We note that in general, we can raise and lower indices for any tensor.

We also need a pointwise inner product for $k$-forms.  By definition, the Riemannian metric $g$ acts on vector fields, but it also induces a metric for $k$-forms. Let $\alpha, \beta$ be two $1$-forms. Then
\[
g(\alpha, \beta)=g(\alpha ^\sharp, \beta ^\sharp),
\]
or if we write in coordinates, then
\[
g(\alpha, \beta)=g^{ij}\alpha_i\beta_j .
\]
For $k$-forms, as well as general covariant $k$-tensors, we have
\[
g(\alpha, \beta)=g^{i_1j_1}g^{i_2 j_2}...g^{i_k j_k}\alpha_{i_1...i_k}\beta_{j_1...j_k}.
\]
Note that for simplicity of notation, we use $g(\cdot, \cdot)$ in all these instances regardless of the type of the input. 

Next, we recall the definition of the Hodge $\star$ operator on forms.  If $v$ is a $k$-form, then $\star v$ is an $(N-k)$-form defined by the following relation
\[
w\wedge \star v =g(w,v)\Vol_M.
\]
 
The $L^2$ scalar product on forms can then be defined by
\be\label{p2b1}
(w,v)=\int_M g(w,v) \Vol_M=\int_M w \wedge \star v.
\ee

We also have for a $k$-form $v$
% \cite[p.101]{DeRhamEng} 
\[
\star \star v=(-1)^{Nk+k}v.
\]
 
In the sequel,  we simply write
\[
\int g(w,v) \quad\mbox{instead of}\quad \int_M g(w,v) \Vol_M.
\]

\subsection{Bochner-Weitzenb\"ock formula}\label{BWPsection}
Recall the Bochner-Weitzenb\"ock formula for $1$-forms relates the Bochner Laplacian, $-\dv \nabla=\nabla^\ast \nabla,$ to the Hodge Laplacian  (see \cite{Taylor_book3})
\be\label{bw1}
\nabla^\ast \nabla \alpha=-\Delta \alpha -\Ric \alpha,
\ee
where $-\Delta$ is the Hodge Laplacian
\[
-\Delta \alpha=\dd \dd^\ast \alpha+ \dd^\ast \dd \alpha,
\]
with 
\be\label{dstardef}
\dd^\ast v=(-1)^{Nk+N+1}\star \dd \star v,
\ee
where $k$ is the degree of $v$, 
and $\Ric$ is the Ricci curvature tensor with one index raised, so $\Ric \alpha$ produces a $1-$form.  More precisely, by definition 
\[
R_{ij}=g^{km} R_{kijm}=R_{kij}^{\quad k},
\]
where $R_{abcd}$ is the Riemann curvature tensor in coordinates.  Then
\[
R^i_j=R^{\;\; i\;\;k}_{k\;j}, 
\]
and the $j$-th coordinate of $\Ric \alpha$ is
\[
(\Ric \alpha)_j=R^l_j \alpha_l=R^{\;\; l\;\;k}_{k\;j}\alpha_l.
\] 
 On a manifold with a constant sectional curvature $-a^2$, this simplifies.  Ricci tensor becomes \cite[Lemma 8.10]{Lee}
\[
R_{ij}=-a^2(N-1)g_{ij},
\]
so $R^{i}_j=-a^2(N-1)\delta^i_j$, and
\[
\Ric \alpha=-a^2(N-1)\alpha.
\]
It follows from \eqref{bw1} that
\be\label{bw2}
\nabla^\ast \nabla \alpha=-\Delta \alpha +a^2(N-1)\alpha.
\ee
For a general $k$-form, one can also relate the Bochner Laplacian to the Hodge Laplacian, but the formula is more complicated.  In coordinates, it is \cite[p.111]{DeRhamEng}
\begin{equation}\label{BWL}
\begin{split}
- \nabla^j\nabla_j \alpha_{i_1...i_k}&=-(\Delta \alpha)_{i_1...i_k}+ \sum_{\nu=1}^k (-1)^\nu R^h_{i_\nu} \alpha_{hi_1\dots \hat i_\nu\dots i_k}\\
&\qquad -2 \sum_{\mu<\nu}^{1\dots k} (-1)^{\mu+\nu} R^{h \ \ \ \ i}_{\ i_\nu i_\mu}\alpha_{ih...\hat{i}_\mu\dots \hat i_\nu\dots i_k} ,
\end{split}
\end{equation}
where $\hat j$ means the index $j$ is not present (we note that we use $-\Delta$ for the Hodge Laplacian as opposed to de Rham, and that following the convention in \cite{Lee}, our curvature tensor is negative of de Rham's.). The terms involving the sums are sometimes referred to as the Weitzenb\"ock curvature.  If $k=1$, \eqref{BWL} becomes \eqref{bw1}.

%\begin{equation}\label{Laplacian}(\Delta \alpha)_{i_1...i_k}
%=- \nabla^j\nabla_j \alpha_{i_1...i_k}+ \sum_{\nu=1}^k (-1)^\nu (\nabla_{i_\nu}\nabla^j - \nabla^j\nabla_{i_\nu})\alpha_{j...\hat{i}_\nu...i_k} ,\end{equation}
Fortunately, if the sectional curvature is constant, \eqref{BWL} can also be simplified. 
% In fact, we have the following lemma, which may very well be known in the literature, but 
% The proof is a straightforward computation, and we include it here for completeness.
\begin{lemma}
Let $k\geq 1$, and $\alpha$ be a smooth $k$-form on $\H$.  Then
\be\label{bwp}
\nabla^\ast \nabla \alpha=-\Delta \alpha +a^2k(N-k) \alpha
\ee
\end{lemma}
\begin{proof}
We work in coordinates. From \cite[Lemma 8.10]{Lee} again we have 
\be\label{ricci1}
R^i_j=-a^2(N-1)\delta^i_j,
\ee
as well as
\be
R_{ijkl}= -a^2(g_{il}g_{jk}-g_{ik}g_{jl}),
\ee
so
\be\label{riemann2}
\begin{split}
R^{h \ \ \ \ i}_{\ i_\nu i_\mu}&=-a^2g^{hi'}g^{il} (g_{i'l}g_{i_\nu i_\mu}-g_{i'i_\mu}g_{i_\nu l})\\
&=-a^2g^{il} (\delta^h_lg_{i_\nu i_\mu}- \delta^h_{i_\mu}g_{i_\nu l})=-a^2(g^{ih}g_{i_\nu i_\mu}- \delta^h_{i_\mu}\delta^i_{i_\nu }).
\end{split}
\ee
We use \eqref{ricci1} in \eqref{BWL} to get that the first sum can be rewritten as
\begin{align*}
 \sum_{\nu=1}^k (-1)^\nu R^h_{i_\nu} \alpha_{hi_1\dots \hat i_\nu\dots i_k}&=-a^2(N-1) \sum_{\nu=1}^k (-1)^\nu \delta^h_{i_\nu} \alpha_{hi_1\dots \hat i_\nu\dots i_k}\\
 &=-a^2(N-1) \sum_{\nu=1}^k (-1)^\nu  \alpha_{i_\nu i_1\dots \hat i_\nu\dots i_k}\\
 &=-a^2(N-1) \sum_{\nu=1}^k (-1)^{2\nu-1}  \alpha_{i_1\dots i_k}\\
 &=a^2(N-1)k \alpha_{i_1\dots i_k},
\end{align*}
where we use the anti-symmetry of $\alpha$ in the third line.  For the second sum we use \eqref{riemann2} to get 
\begin{align*}
 -2 \sum_{\mu<\nu}^{1\dots k} (-1)^{\mu+\nu} R^{h \ \ \ \ i}_{\ i_\nu i_\mu}\alpha_{ih...\hat{i}_\mu\dots \hat i_\nu\dots i_k}=2a^2
 \sum_{\mu<\nu}^{1\dots k} (-1)^{\mu+\nu} (g^{ih}g_{i_\nu i_\mu}- \delta^h_{i_\mu}\delta^i_{i_\nu })\alpha_{ih...\hat{i}_\mu\dots \hat i_\nu\dots i_k}.
 \end{align*}
 We now observe that in the first term, since we are summing with respect to $h$ and $i$ we have by anti-symmetry of $\alpha$ and symmetry of the metric that
 \[
 g^{ih}\alpha_{ih...\hat{i}_\mu\dots \hat i_\nu\dots i_k}=- g^{ih}\alpha_{hi...\hat{i}_\mu\dots \hat i_\nu\dots i_k}=-g^{ih}\alpha_{ih...\hat{i}_\mu\dots \hat i_\nu\dots i_k},
 \]
 so the first term cancels.  We are left with 
 \begin{align*}
 -2a^2 \sum_{\mu<\nu}^{1\dots k} (-1)^{\mu+\nu} \delta^h_{i_\mu}\delta^i_{i_\nu }\alpha_{ih...\hat{i}_\mu\dots \hat i_\nu\dots i_k}
&= 
2a^2 \sum_{\mu<\nu}^{1\dots k} (-1)^{\mu+\nu}\alpha_{i_\mu i_\nu...\hat{i}_\mu\dots \hat i_\nu\dots i_k}\\
 &= 
 2a^2 \sum_{\mu<\nu}^{1\dots k} (-1)^{2\mu+2\nu-1}\alpha_{i_1\dots  i_k}\\
 &= 
- a^2k(k-1)\alpha_{i_1\dots  i_k}.
\end{align*}

\end{proof}

From this we obtain the following corollary that we record here.
\begin{cor}\label{corBW} Let $M$ be a Riemannian manifold of dimension $N$ with a constant sectional curvature $K\in \R$, then if $\alpha$ is a $k$-form we have,
\begin{equation}\label{corBWL}
\begin{split}
  \sum_{\nu=1}^k (-1)^\nu R^h_{i_\nu} \alpha_{hi_1\dots \hat i_\nu\dots i_k} -2 \sum_{\mu<\nu}^{1\dots k} (-1)^{\mu+\nu} R^{h \ \ \ \ i}_{\ i_\nu i_\mu}\alpha_{ih...\hat{i}_\mu\dots \hat i_\nu\dots i_k}=Kk(N-k) \alpha.
\end{split}
\end{equation}
\end{cor}

Another useful Bochner formula is 
\be\label{BochnerLap}
\frac 12\Delta \abs{\alpha}^2=g(\Delta \alpha, \alpha)+g(\nabla \alpha, \nabla \alpha)-a^2k(N-k) g(\alpha,\alpha).
\ee
It follows from Corollary \ref{corBW}, and for example from \cite[Lemma 3.4]{Li_book} .

\subsection{Sobolev space $H^1$ on $\H$} Let $\nabla$ be the Levi-Civita connection on $\H$.  The connection $\nabla$ induces a covariant derivative on any tensor.  If $\alpha$ is a smooth $k$-form, then in particular $\alpha$ is a covariant $k$-tensor, and $\nabla \alpha$ is a $k+1$ covariant tensor.  

We denote by $\nabla^\ast$ the formal adjoint of $\nabla$ defined by
\[
\int g(\alpha, \nabla^\ast \theta) =\int g(\nabla \alpha, \theta),
\]
where $\theta$ is a smooth compactly supported $k+1$ covariant tensor and $\alpha$ is a smooth $k$-form.

We now define weak derivatives.  The definitions are natural generalizations of the Euclidean weak derivatives.
\begin{defn}[Weak $\nabla$] \label{weaknabla}
 Let $\alpha$ be an $L^1_{loc}$ integrable $k$-form, then $\alpha$ is weakly differentiable if there exists some $L^1_{loc}$ covariant $(k+1)$-tensor $\tau$ such that
\begin{equation}
(\alpha, \nabla ^* \theta)=\int g(\alpha, \nabla ^*  \theta ) =\int g(\tau, \theta ) =(\tau, \theta),
\end{equation}
and the above equality holds for any smooth compactly supported covariant $k+1$-tensor $\theta$.
\end{defn}
We can define weak $\dd$ and $\dd^\ast$ in a similar manner.
% we say then $\tau=\nabla\alpha,$ holds weakly.\\ In a similar fashion define $\dd$
\begin{defn}[Weak $\dd$]\label{weakd}
Let $\alpha$ be an $L^1_{loc}$ integrable $k$-form, then $\dd \alpha$ exists in a weak sense if there exists some $L^1_{loc}$  $(k+1)$-form $\tau$ such that
\begin{equation}
(\alpha, \dd ^* \theta) =(\tau, \theta),
\end{equation}
and the above equality holds for any smooth compactly supported $(k+1)$-form $\theta$.
\end{defn}

\begin{defn}[Weak $\dd^\ast$ ]\label{weakdstar}
Let $\alpha$ be an $L^1_{loc}$ integrable $k$-form, then $\dd^\ast \alpha$ exists in a weak sense if there exists some $L^1_{loc}$  $(k-1)$-form $\tau$ such that
\begin{equation}
(\alpha, \dd  \theta) =(\tau, \theta),
\end{equation}
and the above equality holds for any smooth compactly supported $(k-1)$-form $\theta$.
\end{defn}
 
Next we have the inner product  
\begin{equation}
[u, v] = (u, v)+ (\nabla u, \nabla v),
\end{equation}
which induces a norm
\be\label{H1norm}
||u ||=\sqrt{[u ,u]}.
\ee 
With these preparations, the Sobolev space $H^1$ for $\Lambda^k(M)$ is defined as follows.
 
 \begin{defn}[Sobolev space $H^1$ on $k$-forms ]
 $$H^1(\Lambda^k(M))=\overline{\Lambda^k_c(M)}^{H^1},$$
 \end{defn}
 where the closure is taken with respect to the norm given by \eqref{H1norm}, which we from now on denote as 
  $||\cdot ||_{H^1}$.
   
It follows that if $u \in H^1$, then the weak derivative $\nabla u$ exists and belongs to $L^2$.  One can show that using the Bochner-Weitzenb\"ock formula both $\dd u$ and $\dd^\ast u$ exist in a weak sense, and belong to $L^2$.  The proof is exactly the same as in \cite{CC13} except that now we work with $k$-forms instead of $1$-forms.  Therefore, we state it here without proof.

\begin{lemma}\cite[Lemma 2.8]{CC13}\label{importantlemma}
Let $u$ be a $k$-form in $H^1(\H)$. It follows that both weak $\dd u$ and $\dd^\ast u$ exist in the sense of the Definitions \ref{weakd} and \ref{weakdstar}, and belong to $L^2$.  Moreover the following formula holds
\[
\norm{\nabla u}^2_{L^2}=\norm{\dd u}^2_{L^2}+\norm{\dd^\ast u}^2_{L^2}+a^2k(N-k)\norm{u}^2_{L^2}.
\]
\end{lemma} 
\begin{remark}
We point out that in the case of $M=\R^3$ and $u$ being a vector field this reduces to the familiar decomposition 
\[
\norm{\nabla u}^2_{L^2}=\norm{\curl u}^2_{L^2}+\norm{\dv u}^2_{L^2},
\]
since $a=0$,  $\star \dd \alpha=\curl \alpha^\sharp $ and $g(\star \alpha, \star \alpha)=g(\alpha, \alpha)=g(\alpha^\sharp, \alpha^\sharp)$.
\end{remark}
\subsection{Currents}\label{currents}
At some point we will be taking more derivatives that will be guaranteed to exist, so we will need distributional derivatives.  This  brings us to the subject of currents.  Currents can be thought of as distributions acting on compactly supported smooth differential forms.
%   A current $T$ is a linear functional on $\Lambda^k_c(M)$. 
More precisely
\begin{defn} [Currents] \cite[p.34]{DeRhamEng} Let $M$ be an $n-$dimensional manifold, and $\Lambda^k_c(M)$ denote smooth $k$-forms that are compactly supported in $M$.  Then the current $T$ is a linear functional on $\Lambda^k_c(M)$, with the action denoted by
\[
T[\phi],\quad \phi\in \Lambda^k_c(M).
\]
\end{defn}
A relevant example is an analog of a function $f \in L^1_{loc}$ giving a rise to a distribution:  if $\alpha$ is a locally integrable $(n-k)$-form, we can introduce
\be\label{Talpha}
T_\alpha[\phi]=\int_M \alpha \wedge \phi.
\ee
Hence, we can write $\alpha[\phi]$ to denote \eqref{Talpha}.

 Since from  \eqref{p2b1}, the scalar product on forms is given by
\be\label{p2b12}
(w,v)=\int_M g(w,v) \Vol_M=\int_M w \wedge \star v ,
\ee
it follows that
\[
(w,v)=T_w[\star v].
\]
Similarly, a scalar product of a current $T$ with a form $v$ can be defined by  \cite[p.102]{DeRhamEng}
\be\label{p2b}
(T,v)=T[\star v].
\ee
Finally, if $v$ is compactly supported, then  we can define distributional derivatives of $T$ by \cite[p.105]{DeRhamEng}
\be\label{p2b2}
(\dd T,v)=(T,\dd^\ast v), \quad (\dd^\ast T,v)=(T,\dd v).
\ee
We note that these formulas also hold if $T=\alpha$, and $\alpha$ is a smooth form.
 
We will also use the following theorem from \cite{DeRhamEng}.
\begin{thm}\cite[Theorem 17']{DeRhamEng}\label{DeRhamHomologous}
The current $T$ is homologous to zero if and only if $T[\phi]=0$ for all closed smooth forms with compact support.
\end{thm}
In \cite{CC13}, we unwrapped the definitions to show that in the case of a current of degree $1$, this statement is equivalent to 
(recall $\Lambda^l_{c,\sigma}$ denotes smooth, co-closed and compactly supported $l$-forms)
\begin{lemma}
 Let $T$ be a current of degree $1$.
Then $(T,v)=0$ for all
$v \in \Lambda^1_{c,\sigma}(M)$ if and only if $T=\dd P$ for some $0$ degree current $P$.
\end{lemma}
By the same reasoning as in \cite{CC13}, we can show
\begin{lemma}
 Let $T$ be a current of degree $k$.
Then $(T,v)=0$ for all
$v \in \Lambda^k_{c,\sigma}(M)$ if and only if $T=\dd P$ for some $k-1$ degree current $P$.
\end{lemma}

\subsection{The cut-off function and integration by parts} \label{cutoff}
When we integrate/test against anything that has compact support we can use \eqref{p2b2}.  If the integrands do not have compact support, we can multiply one of them by a cut-off function, which we introduce now. First, let
$\phi \in C^{\infty}([0,\infty ))$ satisfy
\begin{itemize}
\item $\chi_{[0,1)} \leq \phi \leq \chi_{[0,2)}$ on $[0,\infty )$, where $\chi_A$ is the characteristic function of the set $A$.
\item $|\phi'| \leq 2$ on $[0,\infty )$.
\end{itemize}
Then, for each $R > 1$, we consider the cut-off function $\phi_R \in C_c^{\infty}(\mathbb{H}^N(-a^2))$ defined by
\begin{equation}
\phi_R (x) = \phi \big ( \frac{\rho(x)}{R}\big ) ,
\end{equation}
where $\rho(x)$ stands for the geodesic distance of $x \in \mathbb{H}^N(-a^2)$ from a preferred reference point $O$ in $\mathbb{H}^N(-a^2)$.

One application will be with the help of the following formula for a vector field $X$ \cite[p.43]{Lee}

\be\label{byparts}
\dv (\phi^2_R X)= \phi^2_R \dv X+ g(\grad  \phi^2_R , X),
\ee
which is the Riemannian analog of the Euclidean formula for a real-valued function $f$, and a vector-valued function $F$
\[
\dv( fF)=f \dv F+\nabla f \cdot F.
\]
Since $\phi_R$ has compact support, when integrated, the left hand side of $\eqref{byparts}$ will go away.

We will also use that since $\abs{\nabla \rho}=1$,
\be\label{dcut}
\abs{\nabla \phi_R}\leq \frac 2R.
\ee

\section{Harmonic $L^2$ k-forms}\label{harmon}
A $k$-form $\alpha$ is harmonic if the Hodge Laplacian of $\alpha$ vanishes. Observe, from the definition of the Hodge Laplacian, that if  $\alpha$ is $\dd$ and $\dd^\ast$ closed, then $\alpha$  must be harmonic. On a compact manifold $M$ with no boundary, the converse can be quickly seen to hold. Indeed, if $-\Delta \alpha =0$, then 
\[
0=\int (-\Delta \alpha, \alpha)=\int_M (\dd \dd^\ast \alpha + \dd^\ast \dd \alpha, \alpha )=\int  (\dd^\ast \alpha, \dd^\ast \alpha)+\int (\dd \alpha, \dd \alpha).
\]
For non-compact manifolds, using cut-off functions, one can extend this result to harmonic forms belonging to $L^2$.  This is the result of Andreotti and Vesentini \cite{AV}.
\begin{thm}\label{AV}
An $L^2$ $k$-form is harmonic if and only if it is $\dd$ and $\dd^\ast$ closed.
\end{thm}

In \cite{Dodziuk_Sobolev}, Dodziuk has studied the cohomology of $L^2$ forms in the context of the Sobolev spaces.  From the statement of \cite[Proposition 2.2]{Dodziuk_Sobolev} one can deduce that a harmonic $L^2$ form $\alpha$  belongs to the Sobolev space $H^{2m}(M)$ for some integer $m$ satisfying $2m>\frac N2-1$.  The integer $m$ is related to the curvature bounds satisfied by $M$. 
 %show that $\alpha$ is in $H^1(M)$.  
 Since the space form $\H$ satisfies these bounds we have that any $L^2$ harmonic $\frac N2$-form on $\H$ belongs to $H^1$ (for $a>0$, so the statement is nontrivial).  

Here we give an alternate proof of that fact and provide an explicit estimate on the norm.  The proof uses \eqref{BochnerLap}.
%Using we can now show that any $L^2$ harmonic form is actually in $H^1$.  
\begin{thm}
Let $\alpha$ be an $L^2$ harmonic $k$-form, then $\alpha$ is in $H^1$, and
\[
\ \int \abs{\nabla \alpha}^2\leq  2a^2k(N-k) \int   \abs{\alpha}^2 .
\]
\end{thm}
\begin{proof}
If $\alpha$ is harmonic, then \eqref{BochnerLap} simplifies to 
\be\label{BochnerLap2}
\frac 12\dv \nabla \abs{\alpha}^2=\abs{\nabla \alpha}^2-a^2k(N-k) \abs{\alpha}^2.
\ee
Integrating the equation against $\phi_R^2$, which is defined in Section \ref{cutoff}, gives

\begin{equation}\label{integrate}
\int \frac{1}{2}( \dv \nabla \abs{\alpha}^{2} )\phi^{2}_{R} = \int  \phi^{2}_{R}\abs{\nabla \alpha}^2
- a^2k(N-k) \int \phi^{2}_{R} \abs{\alpha}^2.
\end{equation}
We now apply \eqref{byparts} (with $X=\nabla  \abs{\alpha}^{2}$ ) to the left hand side to obtain
\begin{align*}
  \int \frac{1}{2}( \dv \nabla |\alpha|^{2} )\phi^2_{R}(\rho)  &=-\frac 12 \int g(\nabla \abs{\alpha}^2,\nabla \phi^2_R)\\
  &\leq \frac 4{R}\int \abs{\nabla \alpha}\abs{\alpha}\phi_R\\
  &\leq \frac 12\int \abs{\nabla \alpha}^2 \phi_R^2 +\frac 8{R^2} \int\abs{\alpha}^2,
%   &=- \int g(\nabla \alpha,\alpha)\phi_{R}(\rho) \phi'_{R}(\rho)\nabla \rho\\
%  & \leq  \norm{\phi_{R}\nabla \alpha }  \norm{\phi'_R\alpha}\\
%  &\leq \frac 12  ||\overline\nabla \nabla F\psi_{k}(\rho)||^{2}+ \frac 12||\alpha\psi'_{k}(\rho) ||^{2}\\
%   &\leq \frac 12  ||\overline\nabla \nabla F\psi_{k}(\rho)||^{2}+ \frac {2}{2^{2k}}||\nabla F ||^{2}
%\
\end{align*}
by \eqref{dcut}, and Cauchy's inequality.
It follows, the right hand side of \eqref{integrate} is bounded by
\[
\int  \phi^{2}_{R}\abs{\nabla \alpha}^2
- a^2k(N-k) \int \phi^{2}_{R} \abs{\alpha}^2 \leq \frac 12\int \abs{\nabla \alpha}^2 \phi_R^2 +\frac 8{R^2} \int\abs{\alpha}^2.
\]
Rearranging and applying the monotone convergence theorem we get
\[
\frac 12 \int \abs{\nabla \alpha}^2\leq  a^2k(N-k) \int   \abs{\alpha}^2,
\]
as needed.

\end{proof}

\section{Hodge Decomposition for general $k$-forms}\label{Hdecomp}

\subsection{Idea of the proof and review of the $L^2$ Hodge Decomposition}\label{mainidea}
We explain the idea of the proof of the $L^2$ Hodge decomposition to motivate the proof we employ for $H^1$.  We follow the presentation in \cite{DeRhamEng} and provide more details.  The main idea  is to use that $L^2$ is a Hilbert space so if we consider the space
\[
X=\overline{\dd\Lambda^{k-1}_c}^{L^2}\oplus \overline{\dd^\ast \Lambda^{k+1}_c}^{L^2},
\]
by definition, this space is closed in $L^2$, so 
\[
L^2=X\oplus X^\perp.
\]
If we can show $X^\perp=\mathcal H^k$, where
again $\mathcal H^k$ denotes the harmonic $L^2$ $k$-forms on the manifold, then the statement of the $L^2$ Hodge decomposition follows. 

  Step 1 is to show that harmonic forms are contained in $X^\perp$.  Step 2 is  to show that the containment holds the other way.  By  Theorem \ref{AV}, a form $\alpha$ in $L^2$ is both closed ($\dd \alpha=0$, irrotational) and co-closed ($\dd^\ast \alpha=0$, divergence free) if and only if it is harmonic.  So if we consider the inner product of a closed form $\alpha$ with $\phi \in 
\dd^\ast \Lambda^{k-1}_c$, then by \eqref{p2b2}
\be\label{l2e2}
(\alpha, \dd^\ast \phi)=(\dd \alpha, \phi)=0,
\ee
since $\alpha$ is closed.  By taking a limit in $L^2$, this can show $\alpha$ is orthogonal to $\overline{\dd^\ast\Lambda^{k+1}_c}^{L^2}$.  One can do a similar proof using $\dd^\ast \alpha=0$ to show $\alpha$ is orthogonal to $\overline{\dd \Lambda^{k-1}_c}^{L^2}$.  This is the idea of Step 1.  For Step 2, we suppose $\alpha \in X^\perp$, and we would like  to show it is both closed and co-closed.  This is done by first considering the inner product of $\dd \alpha$ against $\phi$ that is compactly supported.  So again, by \eqref{p2b2},
\be\label{l2e1}
(\dd \alpha, \phi)=(\alpha, \dd^\ast \phi)=0,
\ee
since $\alpha$ is orthogonal to $X$, it is orthogonal to $\dd^\ast \phi \in \dd^\ast \Lambda^{k+1}_c$.  Because \eqref{l2e1} holds for all compactly supported forms, we have $\dd \alpha=0$ as needed.  One can do a similar computation for $\dd^\ast \alpha$ to show it is equal to zero.  This completes the main idea of the proof in the $L^2$ case.

When we are dealing with $H^1$, even Step 1 is not as quick.  This is because now we have to consider
\[
(\alpha, \dd^\ast \phi)+(\nabla \alpha, \nabla \dd^\ast \phi)
\]
instead of \eqref{l2e2}.  The first term can be handled as before, but we still need to treat the second term.  This is done with the aid of the Bochner-Weitzenb\"ock formula \eqref{bwp}.  We are now ready to begin the proof for $H^1$.

\subsection{Proof of Theorem \ref{Thm1}}

We define the following space
\[
X=\overline{\dd\Lambda^{k-1}_c}^{H^1}\oplus \overline{\dd^\ast \Lambda^{k+1}_c}^{H^1}.
\]
Then $X$ is a complete subspace of the Hilbert space $H^1,$ so it follows that
\[
H^1(\Lambda^k(M))=X \oplus X^\perp,
\]
where $X^\perp$ is defined with respect to the inner product 
\begin{equation}
[u, v] = (u, v)+ (\nabla u, \nabla v).
\end{equation}
The goal is to show$X^\perp=\mathcal H^k$. First we show

\begin{lemma}\label{directsum}
\[
\overline{\dd\Lambda^{k-1}_c}^{H^1}\perp \overline{\dd^\ast \Lambda^{k+1}_c}^{H^1}
\]
\end{lemma}
\begin{proof}
We first observe that if $\dd \phi \in \dd\Lambda^{k-1}_c $ and $\dd^\ast \omega \in \dd^\ast \Lambda^{k+1}_c$, then
\[
[\dd \phi , \dd^\ast \omega]=(\dd \phi , \dd^\ast \omega)+(\nabla \dd \phi ,\nabla \dd^\ast \omega)=(\nabla \dd \phi ,\nabla \dd^\ast \omega),
\]
since by \eqref{p2b2}  and $\dd \dd =0$,
\[
(\dd \phi , \dd^\ast \omega)=(\dd \dd \phi ,  \omega)=0.
\]
Next by \eqref{bwp}, again \eqref{p2b2} and $\dd^\ast \dd^\ast=0$,
\[
(\nabla \dd \phi ,\nabla \dd^\ast \omega)=( \dd \phi ,\nabla ^\ast \nabla \dd^\ast \omega)= ( \dd \phi ,\dd^\ast \dd \dd^\ast \omega)+a^2k(N-k)(\dd \phi, \dd^\ast \omega)=0.
\]

Now let $u \in \overline{\dd\Lambda^{k-1}_c}^{H^1}$ and $v\in \overline{\dd^\ast \Lambda^{k+1}_c}^{H^1}$, then
\begin{align}
u &= \lim_{n\rightarrow \infty} \dd \phi_n \ \ \mbox{for some sequence} \{\dd \phi_n \} \subset \dd\Lambda^{k-1}_c,\\
 v&= \lim_{n\rightarrow \infty} \dd^\ast \omega_n \ \mbox{for some sequence} \{\dd^\ast \omega_n \} \subset \dd^\ast \Lambda^{k+1}_c,
\end{align}
so by the continuity of the inner product
\[
[u, v]=\lim_{n\rightarrow \infty} [ \dd \phi_n, \dd^\ast \omega_n  ]=0
\]

\end{proof}

\begin{lemma}\label{incl1a}
\[
\mathcal H^k \subset X^\perp= (\overline{\dd\Lambda^{k-1}_c}^{H^1}\oplus \overline{\dd^\ast \Lambda^{k+1}_c}^{H^1})^\perp
\]
\end{lemma}
\begin{proof}
We  prove first that if $h$ is a harmonic $L^2$ $k$-form, then it is orthogonal to $\overline{\dd\Lambda^{k-1}_c}^{H^1}.$  
Let $v \in \overline{\dd\Lambda^{k-1}_c}^{H^1}$,  then there exists some sequence $\{\alpha_n \}\subset \Lambda^{k-1}_c$ such that $v=\lim_{n \rightarrow \infty} \dd \alpha_n$, where the limit holds in $H^1$.  It follows that
\[
[h, v]=\lim_{n \rightarrow \infty}[h, \dd \alpha_n]=\lim_{n \rightarrow \infty}\{ (h,\dd \alpha_n)+(\nabla h ,\nabla \dd \alpha_n)\}.
\]
Then for fixed $n$
\be\label{YF}
\begin{split}
[h, \dd \alpha_n] =(h,\dd \alpha_n)+(\nabla h ,\nabla \dd \alpha_n)&= (\dd^\ast h, \alpha_n)+(\nabla^\ast \nabla h ,  \dd \alpha_n)\\
 &= (\nabla^\ast \nabla h ,  \dd \alpha_n),
 \end{split}
 \ee
 since $h$ is $\dd^\ast$-closed by Theorem \ref{AV}.  Continuing and using \eqref{bwp} 
 we have
 \be\label{YF2}
\begin{split}
[h, \dd \alpha_n] &= a^2k(N-k) ( h , \dd \alpha_n)\\
 &= a^2k(N-k)( \dd^\ast h , \alpha_n)\\
 &=0.
 \end{split}
\ee
using again that $h$ is $\dd^\ast$-closed.  Taking the limit in \eqref{YF2} then implies 
\[
[h, v]=0.
\]
Next we show  $\mathcal H^k\subset \overline{\dd^\ast \Lambda^{k+1}_c}^{H^1}$.  Similarly as above, it is enough to show
\[
[h, \dd^\ast \omega]=0,
\]
for $ \omega \in \Lambda^{k+1}_c$.  To that end
\be
\begin{split}
[h, \dd^\ast \omega]=(h,\dd^\ast \omega)+(\nabla h ,\nabla \dd ^\ast\omega)=(\dd h,\omega)+(\nabla h ,\nabla \dd ^\ast\omega)
 &= (\nabla^\ast \nabla h ,  \dd^\ast \omega),
 \end{split}
 \ee
 since $h$ is $\dd$-closed by Theorem \ref{AV}.  Using again  \eqref{bwp} we get
\[
[h, \dd^\ast \omega]= a^2k(N-k) ( h ,\dd^\ast \omega)=  a^2k(N-k)( \dd h , \omega)=0
\]
as needed.

 \end{proof}

Hence we need to show
 \begin{prop}\label{incl2}
\be
X^\perp \subset \mathcal H^k.
\ee
\end{prop}
\begin{proof}
Let $u \in X^\perp\subset H^1$.  We show $u$ is a harmonic $k$-form by showing $$\dd u =0= \dd^\ast u.$$  Then $u$ must be harmonic by Theorem $\ref{AV}$.

We proceed as follows.
First, let $\theta$ be a smooth, compactly supported $k-1$ form, and consider $(\dd^\ast u, \theta ).$ Note that we could view $\dd^\ast u$ as a distributional derivative, but in this case, since $u \in H^1$, by Lemma \ref{importantlemma}, this is actually a weak $\dd^\ast$, so by Definition \ref{weakdstar} 
\be\label{dstar1}
(\dd^\ast u, \theta ) =(u, \dd\theta),
\ee
since $\theta$ has compact support.  Next,  observe $\dd\theta \in X$, which means
\[
[u,\dd \theta] = 0.
\]
So by the definition of the inner product,  \eqref{dstar1}, and Definition \ref{weaknabla} of the weak covariant derivative $\nabla u$, we get
 \be\label{dstar2}
 (\dd^\ast u, \theta )=-(\nabla u, \nabla d \theta)=-( u, \nabla^\ast \nabla \dd \theta).
 \ee
 Then  by \eqref{bwp}
 \be 
\nabla^*\nabla \dd\theta = \dd \dd^* \dd \theta + a^2 k(N-k) \dd\theta.
\ee
Plugging into \eqref{dstar2}, we obtain
\be
(\dd^\ast u, \theta )=-( u, \dd \dd^* \dd \theta) - a^2 k(N-k)(u, \dd\theta).
\ee
We apply now \eqref{p2b2} to deduce  
$\dd^\ast \dd \dd^\ast  u+(a^2 k(N-k)+1)\dd^\ast u$, as a current, when tested against a compactly supported smooth form gives $0$.  This means
\be\label{dstareq}
\dd^\ast \dd \;\dd^\ast  u+(a^2 k(N-k)+1)\dd^\ast u=0,
\ee
in a sense of currents (distributions).  Moreover, if we let $h=\dd^\ast u$, then \eqref{dstareq} becomes
\be\label{dstareq2}
\dd^\ast \dd h+(a^2 k(N-k)+1)h=0,
\ee and by definition of $h$,  $-\Delta h=\dd^\ast \dd h $.  It follows, $h$ solves an elliptic equation  
\[
-\Delta h+(a^2k(N-k)+1)h=0,
\]
so by elliptic regularity, $h$ is in fact a smooth $k-1$ form.  We now test \eqref{dstareq2} against $\phi_R^2 h$, and we get
\be\label{eq1}
(\dd^\ast \dd h, \phi_R^2 h)+(a^2 k(N-k)+1)(h, \phi_R^2 h)=0.
\ee
We integrate by parts the first term to obtain
\be\label{eq2}
( \dd h, \dd \phi^2_R \wedge h )+( \dd h, \phi^2_R\dd h)=2( \dd h, \phi_R\dd \phi_R\wedge h)+( \dd h, \phi^2_R\dd h).
\ee
Combining \eqref{eq1} and \eqref{eq2}, we have
\be\label{eqlast}
 ( \dd h, \phi^2_R\dd h)+(a^2 k(N-k)+1)(h, \phi_R^2 h)\leq \left| 2( \dd h, \phi_R\dd \phi_R \wedge h)\right|,
\ee
and Cauchy's inequality when applied to the right hand side gives us 
\[ \left| 2( \dd h, \phi_R\dd \phi_R\wedge h)\right|\leq \frac 1 2 \int \phi^2_R| \dd h|^2 + \frac {8}{R^2}\int |h|^2,
\] where we used $\abs{\nabla \phi_R}=|\dd\phi_R |\leq 2/R$ by \eqref{dcut}. Inserting this into \eqref{eqlast} we obtain

\[
\frac 12  \int\phi_R^2\abs{\dd h}^2+(a^2 k(N-k)+1)\int \phi_R^2 \abs{h}^2 \leq \frac {8}{R^2}\int \abs{h}^2.
\]
Now, by Lemma \ref{importantlemma}, $\int\abs{h}^2=\int\abs{\dd^\ast u}^2< \infty$ so by letting $R\rightarrow\infty$, the right hand side goes to zero, and gives
 us \[  \norm{h}_{L^2}= 0,
\]
(and $\norm{\dd h}=0$) as needed.

Next we show $\dd u=0$.  Similarly, let $\alpha$ be a smooth, compactly supported $k+1$ form, and consider 
\be\label{dstar3}
(\dd u, \alpha ) =(u, \dd^\ast\alpha),
\ee
since $\alpha$ has compact support.  Next, $\dd^\ast \alpha$ is co-closed, so $\dd^\ast\alpha \in X$, which means
\[
[u,\dd ^\ast \alpha] = 0.
\]
 \be\label{dstar4}
 (\dd u, \alpha )=-(\nabla u, \nabla \dd^\ast \alpha)=-( u, \nabla^\ast \nabla \dd ^\ast \alpha).
 \ee
 By \eqref{bwp}
 \be 
\nabla^*\nabla \dd^\ast\alpha = \dd^\ast \dd \dd^\ast \alpha + a^2 k(N-k) \dd^\ast\alpha.
\ee
Plugging into \eqref{dstar4}, we obtain
\be
(\dd u, \alpha )=-( u, \dd^\ast \dd \dd^\ast \alpha) - a^2 k(N-k)(u, \dd^\ast\alpha).
\ee
So again, by the definition of distributional derivatives of a current we have

$\dd\dd^\ast \dd  u+(a^2 k(N-k)+1)\dd u$, as a current, when tested against a compactly supported $k+1$ smooth form gives $0$.  This means
\be\label{dstareq3}
\dd \dd^\ast  \dd u+(a^2 k(N-k)+1)\dd u=0,
\ee
as a current. Now, let $\omega=\dd u$, then \eqref{dstareq3} becomes
\[
\dd \dd^\ast  \omega+(a^2 k(N-k)+1)\omega=0,
\] or equivalently, since $\dd \omega=0$,
\[
-\Delta \omega+(a^2 k(N-k)+1)\omega=0.
\]
Again, the elliptic regularity tells us that $\omega$ is a smooth $k+1$ form.  We now integrate the above equation against $\phi_R^2 \omega$, and we get
\be\label{eq3}
(\dd\dd^\ast \omega, \phi_R^2 \omega)+(a^2 k(N-k)+1)(\omega, \phi_R^2 \omega)=0.
\ee\\
We use \eqref{p2b2} to move $\dd$ in the first term onto $\phi^2_R\omega$, and then apply the formula \eqref{dstardef} to produce the following expression
\be\label{eq4}
\begin{aligned}\dd^\ast (\phi^2_R\omega) =   (-1)^{N(k +1) +N +1} \ast\dd\ast(\phi^2_R\omega)=& (-1)^{kN +2N +1}(\dd\phi^2_R\wedge\ast\omega)+ \phi^2_R\dd^\ast\omega \\  = & (-1)^{kN +2N +1}2 \phi_R \dd \phi_R\wedge \ast\omega+ \phi_R ^2 \dd^\ast \omega.
\end{aligned}\ee\\
Then \eqref{eq3} and \eqref{eq4} give
\[ ( \dd^\ast \omega, \phi^2_R\dd^\ast \omega)+(a^2 k(N-k)+1)(\omega, \phi_R^2 \omega)\leq \left| 2( \dd^\ast \omega , \phi_R\dd \phi_R\wedge \ast \omega)\right|,
\]
so just like before, using Cauchy's inequality to the right hand side, we obtain
\[ \left| 2( \dd^\ast\omega , \phi_R\dd \phi_R\wedge \ast\omega)\right|\leq \frac 12 \int\phi_R^2\abs{\dd^\ast \omega}^2 + \frac {8}{R^2}\int \abs{\omega}^2,
\] where we used $(\ast\omega,\ast \omega)=(\omega,\omega)$. Combining with the last inequality we arrive at

\[
 \frac 12\int\phi_R^2\abs{\dd^\ast\omega}^2+(a^2 k(N-k)+1)(\omega, \phi_R^2 \omega)\leq \frac {8}{R^2}\int \abs{\omega}^2,
\]
which by Lemma \ref{importantlemma} allows us to conclude by taking the limit that  $\norm{\omega}_{L^2}=0$ as needed.

\end{proof}

To finish the proof of the theorem we need to show \eqref{hodgepal} holds.  To that end we prove the following lemmas

\begin{lemma}\label{dplemma}
Let $v \in \overline{\dd\Lambda^{k-1}_c}^{H^1}$, then
\be\label{dp}
v=\dd \beta
\ee
for some $(k-1)$-current $\beta$.
\end{lemma}
\begin{proof} We would like to apply \ref{DeRhamHomologous}.  To that end we need to consider the inner product of $v$ with $\theta$ where $\theta$ is compactly supported and co-closed.  Since $v \in \overline{\dd\Lambda^{k-1}_c}^{H^1}$, then 
 $v=\lim_{n\rightarrow \infty} \dd \beta_n$, with $\beta_n \in \Lambda_{c}^{k-1}$, where the limit is in the $H^1$ norm.  This implies that the limit also holds in $L^2$, which in turn implies 
\[
(v,\theta)=\lim_{n \rightarrow \infty}(\dd \beta_n, \theta)=\lim_{n \rightarrow \infty}(\beta_n, \dd^\ast \theta)=0,
\]
since $\theta$ is co-closed.  So the result follows.
 
 \end{proof}

\begin{lemma}\label{dstarlemma}
Let $v\in  \overline{\dd^\ast \Lambda^{k+1}_c}^{H^1}$, then
\be\label{dstaromega}
v=\dd^\ast \omega
\ee
for some $(k+1)$-current $\omega$.
\end{lemma}
\begin{proof}
The idea is as in the previous lemma, but now we consider the inner product of $\star v$ with co-closed, compactly supported $\theta$. Similarly, from the convergence in $H^1$ we can deduce (we do not keep track of the $\pm$ signs coming from $\dd^\ast$ and $\star \star$)
\[
(\star v, \theta)=\lim_{n \rightarrow \infty}(\star \dd^\ast \omega_n, \theta)=\pm \lim_{n \rightarrow \infty}(\star \star \dd \star \omega_n, \theta)= \pm \lim_{n \rightarrow \infty}(\dd \star \omega_n, \theta)= \pm \lim_{n \rightarrow \infty}( \star \omega_n,\dd^\ast \theta)= 0.
\]
It follows from Theorem \ref{DeRhamHomologous} that $\star v=\dd  S$ for some current $ S$.  We now let $\omega=\pm \star S$ to obtain
\[
\star v=\pm\dd \star \omega,
\]
so that 
\[
v=\pm \star \dd \star \dd \omega=\dd^\ast \omega,
\]
as needed.
 \end{proof}

Equation  \eqref{hodgepal} now follows from Lemmas \ref{dplemma}, \ref{dstarlemma} and the equation \eqref{hodgep}.

\section{Proof of Theorem \ref{Thm2}}\label{ProofThm2}

The theorem follows from the following lemma.
 
\begin{lemma}\label{dstarw_l}
\[
\tV=  \overline{\dd^\ast \Lambda^2_c}^{H^1}
\]
\end{lemma}
\begin{proof}
By definition,  $\overline{\dd^\ast \Lambda^2_c}^{H^1}\subset \tV$. To show the inclusion holds the other way, we give a constructive proof. 
 
If $v\in \tV$, then $$v=\lim_{n\rightarrow \infty} v_n,$$ where $v_n \in \Lambda_{c,\sigma}^1$ and the limit is in $H^1$.  Now, for each $v_n$, we have $$\dd^* v_n=0,$$ so $\star \dd \star v_n=0$.  Which means that $\dd\star v_n=0$, so $\star v_n=\dd f_n$, where $f_n$ is some (smooth) function.  Here we use that $\mathbb H^2(-a^2)$ is simply connected, and that we are in two dimensions, so  $\star v_n$ is a $1-$ form.  Then
\[
v_n=-\star \dd f_n.
\]
However, $f_n$ may or may not be compactly supported in 
$\mathbb{H}^2(-a^2)$.    On the other hand, since $v_n$ is compactly supported in $\mathbb{H}^2(-a^2)$, it follows that $\dd f_n = \star v_n$ is also compactly supported in $\mathbb{H}^2(-a^2)$. So, we can take some sufficiently large $R_n$ for which we have
\begin{equation*}
\supp \dd f_n \subset \overline{B(R_n)} ,
\end{equation*}   
and hence the following relation holds.
\begin{equation*}
\dd f_n \big |_{\Omega (R_n)} = 0 , 
\end{equation*}
where $\Omega (R_n) = \{  x \in \mathbb{H}^2(-a^2) : \rho (x) > R_n \}$ .
Since the exterior domain $\Omega (R_n)$ is path-connected, there exists a constant $C_n$ such that
\begin{equation*}
f_n \big |_{\Omega (R_n)} = C_n .
\end{equation*}
We  now consider the new function $\widetilde{f}_n = f_n -C_n$.  It follows that $\tilde{f}_n$ satisfies
\begin{equation*}
\begin{split}
 \supp \widetilde{f}_n \subset \overline{B(R_n)}\qand
 \star v_n = \dd \widetilde{f}_n . 
\end{split}
\end{equation*}
This means that we can replace $f_n$ by $\widetilde{f}_n \in C_c^{\infty}(\mathbb{H}^2(-a^2))$ in our analysis, and write
\[
v_n=-\star \dd \tilde f_n.
\]
We next connect $\star \dd \widetilde{f}_n$ to some $2$-form, $\omega_n$, so that, $v_n=-\star \dd \widetilde{f}_n=\dd^\ast \omega_n$. Let $$\omega_n=\widetilde{f}_n \Vol_{\mathbb{H}^2(-a^2)}.$$  Then
\[
\dd^\ast \omega_n=-\star \dd \star \omega_n=-\star \dd \star \widetilde{f}_n \Vol_{\mathbb{H}^2(-a^2)}= -\star \dd  \widetilde{f}_n=v_n.
\]
Finally, notice that $\omega_n \in \Lambda^2_c(\mathbb{H}^2(-a^2))$ as desired.
 \end{proof}

\bibliography{ref}
\bibliographystyle{plain}

\end{document}